\documentclass[12pt]{amsart}
\usepackage[T1]{fontenc}
\usepackage{lmodern}
\usepackage{ifthen}
\usepackage{amsfonts}
\usepackage{amsxtra}
\usepackage{amssymb}
\usepackage{amsthm}
\usepackage{array}
\usepackage[margin=1in]{geometry}
\usepackage{xcolor}
\definecolor{cite}{rgb}{0.50,0.00,1.00}
\definecolor{url}{rgb}{0.00,0.50,0.75}
\definecolor{link}{rgb}{0.00,0.00,0.50}
\usepackage[colorlinks,linkcolor=link,urlcolor=url,citecolor=cite,pagebackref,breaklinks]{hyperref}
\usepackage{mathtools}
\usepackage{mathrsfs}
\usepackage[all]{xy}
\usepackage[lite,abbrev,msc-links]{amsrefs}
\begin{document}
\title[DR-indecomposable special fibers of semi-stable reductions]{Some examples of DR-indecomposable special fibers of semi-stable reductions over Witt rings}
\author[Mao Sheng]{Mao Sheng}
\email{msheng@ustc.edu.cn}
\address{School of Mathematical Sciences,
	University of Science and Technology of China, Hefei, 230026, China}
\author[Junchao Shentu]{Junchao Shentu}
\email{stjc@ustc.edu.cn}
\address{School of Mathematical Sciences,
	University of Science and Technology of China, Hefei, 230026, China}


\theoremstyle{definition}
\newtheorem{Unity}{Unity}[section]
\newtheorem*{defn*}{Definition}
\newtheorem{defn}[Unity]{Definition}
\newtheorem{Notation}[Unity]{Notation}
\newtheorem{Claim}[Unity]{Claim}
\newtheorem*{Setting}{Setting}

\theoremstyle{plain}
\newtheorem*{thm*}{Theorem}
\newtheorem{thm}[Unity]{Theorem}
\newtheorem{prop}[Unity]{Proposition}
\newtheorem*{prop*}{Proposition}
\newtheorem*{cor*}{Corollary}
\newtheorem{cor}[Unity]{Corollary}
\newtheorem{lem}[Unity]{Lemma}
\newtheorem{conj}[Unity]{Conjecture}
\newtheorem{prob}[Unity]{Problem}
\newtheorem{question}[Unity]{Question}
\theoremstyle{remark}
\newtheorem*{rmk*}{Remark}
\newtheorem{rmk}[Unity]{Remark}
\newtheorem{exmp}[Unity]{Example}

\numberwithin{Unity}{section}

\newcommand{\sA}{\mathscr{A}}
\newcommand{\sB}{\mathscr{B}}
\newcommand{\sC}{\mathscr{C}}
\newcommand{\sD}{\mathscr{D}}
\newcommand{\sE}{\mathscr{E}}
\newcommand{\sF}{\mathscr{F}}
\newcommand{\sG}{\mathscr{G}}
\newcommand{\sH}{\mathscr{H}}
\newcommand{\sI}{\mathscr{I}}
\newcommand{\sJ}{\mathscr{J}}
\newcommand{\sK}{\mathscr{K}}
\newcommand{\sL}{\mathscr{L}}
\newcommand{\sM}{\mathscr{M}}
\newcommand{\sN}{\mathscr{N}}
\newcommand{\sO}{\mathscr{O}}
\newcommand{\sP}{\mathscr{P}}
\newcommand{\sQ}{\mathscr{Q}}
\newcommand{\sR}{\mathscr{R}}
\newcommand{\sS}{\mathscr{S}}
\newcommand{\sT}{\mathscr{T}}
\newcommand{\sU}{\mathscr{U}}
\newcommand{\sV}{\mathscr{V}}
\newcommand{\sW}{\mathscr{W}}
\newcommand{\sX}{\mathscr{X}}
\newcommand{\sY}{\mathscr{Y}}
\newcommand{\sZ}{\mathscr{Z}}
\newcommand{\spec}{\textrm{Spec}}
\newcommand{\dlog}{\textrm{dlog}}

\newcommand{\A}{{\mathbb A}}
\newcommand{\B}{{\mathbb B}}
\newcommand{\C}{{\mathbb C}}
\newcommand{\D}{{\mathbb D}}
\newcommand{\E}{{\mathbb E}}
\newcommand{\F}{{\mathbb F}}
\newcommand{\G}{{\mathbb G}}
\renewcommand{\H}{{\mathbb H}}
\newcommand{\I}{{\mathbb I}}
\newcommand{\J}{{\mathbb J}}
\renewcommand{\L}{{\mathbb L}}
\newcommand{\M}{{\mathbb M}}
\newcommand{\N}{{\mathbb N}}
\renewcommand{\P}{{\mathbb P}}
\newcommand{\Q}{{\mathbb Q}}
\newcommand{\Qbar}{\overline{\Q}}
\newcommand{\R}{{\mathbb R}}
\newcommand{\SSS}{{\mathbb S}}
\newcommand{\T}{{\mathbb T}}
\newcommand{\U}{{\mathbb U}}
\newcommand{\V}{{\mathbb V}}
\newcommand{\W}{{\mathbb W}}
\newcommand{\Z}{{\mathbb Z}}

\thanks{This work was partially supported by National Natural Science Foundation of China (Grant No. 11622109, No. 11626253) and the Fundamental Research Funds for the Central Universities.}

\maketitle

\begin{abstract}
	We answer negatively an open problem of Illusie on the DR-decomposability of the log de Rham complex of the special fiber of a semi-stable reduction over the Witt ring. We also show that $E_1$ degeneration of the Hodge to log de Rham spectral sequence does not imply DR-decomposability of semi-stable varieties. 
\end{abstract}

\section{Introduction}
The work of Deligne-Illusie \cite{Del_Ill1987} is fundamental in Hodge theory since it gives a new method to establish the $E_1$-degeneration property of the Hodge to de Rham spectral sequence. Let $k$ be a perfect field of positive characteristic and $X_0$ an algebraic variety over $k$. We have the following commutative diagram of Frobenius
$$\xymatrix{
  X_0 \ar[r]^{F=F_{X_0/k}}\ar[rd] & X_0'\ar[r]^{\pi} \ar[d] &X_0 \ar[d]\\
&   {\rm Spec}\ k \ar[r]^{\sigma}&{\rm Spec}\ k,\\
}$$
The variety $X_0$ is said to be \emph{DR-decomposable} if the complex $\tau_{<p}F_{\ast}\Omega^{\bullet}_{X_0}$ is quasi-isomorphic to $\bigoplus_{i=0}^{\dim X}\Omega^i_{X'_0/k}[-i]$, where $\Omega^{\bullet}_{X_0}$ is the de Rham complex of $X_0/k$. The main result of Deligne-Illusie asserts that for \emph{smooth} varieties, $X_0$ is $W_2=W_2(k)$-liftable if and only if  it is DR-decomposable. On the other hand, if $X_0$ is proper over $k$ and $\dim X_0<p$, the DR-decomposability of $X_0$ implies the $E_1$-degeneration of the Hodge to de Rham spectral sequence (for $\dim X_0=p$, the $E_1$-degeneration also holds by the Grothendieck duality). Properness on $X_0$ is required because of finite dimensionality of Hodge cohomologies. However, it is not clear whether one can remove the assumption on the dimension of $X_0$: this is exactly one of the two open problems posed by Illusie \cite{Illusie2002}. It is neither clear whether $E_1$-degeneration would imply the DR-decomposability. \\

This note grew out from our study on the other open problem posed by Illusie in loc. cit, that is about the generalization of Deligne-Illusie's main result to semi-stable varieties over $k$. Note that semi-stable varieties appear naturally in algebraic geometry as very typical singular varieties.  The problem is stated as follows: let $k$ be as above and $W=W(k)$ the ring of Witt vectors. For a semi-stable reduction $X$ over $W$, we set $X_0=X\times_{W}k$, the special fiber of $X$, and $F: X_0\to X'_0$ the relative Frobenius. Consider the complex of $\sO_{X_0}$-modules:
$$\Omega^{\log\bullet}_{X_0}=\Omega^{\bullet}_{X}(\textrm{log}X_0)|_{X_0}.$$
\begin{prob}[Illusie,  Problem 7.14 \cite{Illusie2002}]\label{prob_Illusie}
Is the complex $\tau_{<p}F_{\ast}\Omega^{\log\bullet}_{X_0}$
decomposable in $D(X'_0)$?
\end{prob}
Our answer to this problem is NO.  Indeed,  we constructed explicit examples of semi-stable reductions over $W$ negating the problem, whose dimension can be arbitrary large (in the curve case the answer is affirmative for cohomological reason) and the characteristic of $k$ can be arbitrary. See \S3 for the construction. We also examined the $E_1$-degeneration property of these examples. It turns out that all examples we constructed whose dimensions are less than or equal to the characteristic of the residue field have the $E_1$-degeneration property. This is a direct consequence of Theorem \ref{thm_degeneration_E1} and Deligne-Illusie's decomposation theorem. Therefore, the $E_1$-degeneration property is NOT equivalent to the DR-decomposability in the semi-stable (non-smooth) case. We are not aware of similar results in the smooth case.

\section{DR-decomposability and log deformation}
We use the log geometry as developed in the work \cite{KKato1988} to study Problem \ref{prob_Illusie}, and the construction of our examples is mainly based on a simple criterion of the DR-decomposability in terms of the existence of a log smooth deformation over the log scheme $(W_2(k),1\mapsto 0)$ (Theorem \ref{thm_decom_lifting}).\\

Let $X$ be a semi-stable reduction over $W$. Let $M_{X_0}$ (resp. $M_{\spec(k)}$) be the log structure on $X$ (resp. $\spec(W)$) attached to the reduced normal crossing divisor $X_0$ (resp. $\spec(k)$) (Example (1.5) \cite{KKato1988}). Then the extended morphism of log schemes $f: (X,M_{X_0})\to (\spec(W),M_{\spec(k)})$  is smooth. Let $(X_0,M_0)\to {\bf k}:=(k, 1\mapsto 0)$ be the base change of $f$ via the inclusion $\spec(k)\to \spec(W)$. When the context is clear, we denote the log scheme $(X_0,M_0)$ simply by $X_0$ (in some other occasion, we use $\underline{X}$ to denote the underlying scheme of a log scheme $X$).  It is known that the morphism $X_0\to {\bf k}$ is smooth, and the de Rham complex $\Omega^{\bullet}_{X_0/\bf{k}}$ of the log variety $X_0/\bf{k}$ is naturally isomorphic to the complex $\Omega^{\log\bullet}_{X_0}$ considered in \S1  (1.7 \cite{KKato1988}). Moreover, it is known that the log structure $M_0$ of $X_0$ is of semi-stable type:
\begin{defn}(\cite{Ollson2003})\label{defn_semistable_type}
A log variety $X$ over $\bf k$ is called semi-stable type if \'etale locally over each closed point $x\in \underline{X}$ it is strict smooth over $$(\spec(k(x)[x_1,\cdots,x_r]/(x_1\cdots x_r)),\bigoplus_{i=1}^r \mathbb{N}e_i,e_i\mapsto x_i),$$
where the log structure is induced by the homomorphism of monoids $\bigoplus_{i=1}^r \mathbb{N}e_i\rightarrow\sO_{\underline{X}}$ defined by $e_i\mapsto x_i$.
\end{defn}
Let $F$ be the absolute Frobenius of the log scheme $\mathbf{k}$ which is given by the commutative diagram
$$\xymatrix{
	k\ar[r]^{F_{k}} & k\\
	\mathbb{N} \ar[u]^0 \ar[r]^{\times p} & \mathbb{N}\ar[u]^0.\\
}$$
It is easy to verify that $F$ is liftable to the log scheme ${\bf W}_2:=(W_2, 1\mapsto 0)$ (but not to the log scheme $(W_2, 1\mapsto p)$!), and an obvious lifting $G$ over ${\bf W}_2$ is given by the following commutative diagram:
$$\xymatrix{
	W_2\ar[r]^{F_{W_2}} & W_2\\
	\mathbb{N} \ar[u]^0 \ar[r]^{\times p} & \mathbb{N}\ar[u]^0,\\
}$$
where $F_{W_2}$ is the Frobenius automorphism of $W_2$. A special case of the Kato's decomposition theorem is the following
\begin{thm}[Theorem 4.12 \cite{KKato1988}]\label{Kato decomposition}
Let $X/\bf k$ be a log variety of semi-stable type and $X'$ the base change of $X$ via the absolute Frobenius of $\bf k$.  Let $F_{X/\bf k}: X\to X'$ be the relative Frobenius. Then the complex $\tau_{<p}F_{X/{\bf k}\ast}\Omega^{\bullet}_{X/\bf k}$ is decomposable if and only if $X'$ is liftable to ${\bf W}_2$.
\end{thm}
Remark that Kato's decomposition theorem works for a log variety of Cartier type which is more general than semi-stable type (Definition 4.8 \cite{KKato1988}). In the following, we show further that $X'$ is liftable to ${\bf W}_2$ if and only if $X$ itself is liftable to ${\bf W}_2$, and hence we obtain the following criterion for DR-decomposability:
\begin{thm}\label{thm_decom_lifting}
Notation and assumption as Theorem \ref{Kato decomposition}. Then
the complex $\tau_{<p}F_{X/{\bf k}\ast}\Omega^{\bullet}_{X/\bf k}$ is decomposable if and only if $X$ is liftable to ${\bf W}_2$.
\end{thm}
\begin{proof}
Via the base change by $G$, one obtains a $\mathbf{W}_2$-lifting of $X'$ from that of $X$.  Since $G$ is not an isomorphism of log schemes, our argument is to show the converse nevertheless is still true. Let $\omega_{X}\in H^2(X,T_{X/\mathbf{k}})$ (resp. $\omega_{X'}\in H^2(X',T_{X'/\mathbf{k}})$) be the obstruction class of the lifting of $X$ (resp. $X'$) to $\mathbf{W}_2$. Recall that $\omega_{X}$ is constructed as follows: Let $\{U_i\}$ be an affine cover of $X$. Choosing for each $U_i$ a log smooth lifting $\sU_i$ on $\mathbf{W}_2$, we then have that on each overlap $U_{ij}=U_i\cap U_j$ there exists an isomorphism $\alpha_{ij}:\sU_j|{U_{ij}}\rightarrow \sU_i|{U_{ij}}$. Then $\omega_{X}$ is represented by $\{(U_i\cap U_j\cap U_k, \alpha_{ij}\alpha_{jk}\alpha_{ki})\}$. Because of the existence of $G$, $\{(G^{-1}(U_i\cap U_j\cap U_k), G^\ast\alpha_{ij}\alpha_{jk}\alpha_{ki})\}$ represents $\omega_{X'}$. Thus we have that  $\sigma^\ast(\omega_{X})=\omega_{X'}$ through the canonical map
$$\xymatrix{
H^2(X',T_{X'/\mathbf{k}}) \ar@{}[r]|= & H^2(X',\sigma^\ast T_{X/\mathbf{k}})\\
 & H^2(X,T_{X/\mathbf{k}})\ar[u]^{\sigma^\ast}
}.$$
The above equality uses the fact that $\sigma^*\Omega_{X/\mathbf{k}}=\Omega_{X'/\mathbf k}$ and that both sheaves are locally free (see 1.7 and Proposition 3.10 \cite{KKato1988}). However, since $\sigma$ is an isomorphism of schemes, the map $\sigma^*$ in the vertical line is bijective.  It follows immediately that $\omega_{X}=0$ under the assumption that $X'$ is $\mathbf{W}_2$-liftable and hence $X$ itself is $\mathbf{W}_2$-liftable.
\end{proof}
The following corollary ensures it is valid to assume $k$ is algebraically closed in the study of Problem \ref{prob_Illusie}.
\begin{cor}\label{cor_change_field}
Let $f:X\rightarrow \mathbf{k}$ be a smooth morphism of semistable type and $k'$ be a perfect field containing $k$. Denote by $\mathbf{k'}$ the field $k'$ with the induced log structure from $\mathbf{k}$ and by $X_{\mathbf{k'}}$ the log base change. Then $\tau_{<p}F_{X/{\bf k}\ast}\Omega^{\bullet}_{X/\mathbf{k}}$ is decomposable if and only if $\tau_{<p}F_{X_{\mathbf{k'}}/{\bf k'}\ast}\Omega^{\bullet}_{X_{\mathbf{k'}}/{\mathbf{k'}}}$ is decomposable.
\end{cor}
\begin{proof}
By Theorem \ref{thm_decom_lifting}, it is enough to show that a $(W_2(k'), \mathbb{N}\mapsto 0)$-lifting of $X_{\mathbf{k'}}$ induces a $(W_2(k), \mathbb{N}\mapsto 0)$-lifting of $X$. By the flat base change, one has the isomorphism $H^2(X,T_{X/\mathbf{k}})\otimes_kk'=H^2(X_{\mathbf{k'}}, T_{X/\mathbf{k'}})$ and hence the injection $\alpha: H^2(X,T_{X/\mathbf{k}})\to H^2(X_{\mathbf{k'}}, T_{X/\mathbf{k'}})$. Then, by the same arguments in Theorem \ref{thm_decom_lifting}, the obstruction class $ob_k$ to lifting $X$ to $\mathbf{W_2(k)}$ is mapped to to the obstruction class $ob_{k'}$ of lifting $X_{\mathbf{k'}}$ to $(W_2(k'),\mathbb{N}\mapsto 0)$ via the map $\alpha$. By the condition that $\alpha(ob_k)=ob_{k'}=0$, it follows that $ob_k=0$.
\end{proof}

\begin{rmk}
After presenting our results, Weizhe Zheng provided us a more conceptual proof of Theorem \ref{thm_decom_lifting}: Denote by $\textrm{Lift}(X)$ (resp.$\textrm{Lift}(X')$) the groupoid of liftings of $X$ (rsep. $X'$) over $\mathbf{W}_2$. Let $G:\mathbf{W}_2\rightarrow\mathbf{W}_2$ be a lifting of the log Frobenius morphism $F:\mathbf{k}\rightarrow\mathbf{k}$. Given a lifting $X^{(1)}\in \textrm{Lift}(X)$, the pullback of $X^{(1)}$ along $G$ gives an object in $\textrm{Lift}(X')$. With the obvious assignments on morphisms, one can get a functor
$$
A: \textrm{Lift}(X)\rightarrow\textrm{Lift}(X').
$$
Conversely, let $X'^{(1)}\in \textrm{Lift}(X')$ be a lifting of $X'$. Denote by $i:X'\hookrightarrow X'^{(1)}$ the canonical strict closed immersion and by $\sigma:X'\rightarrow X$ the base change of $F:\mathbf{k}\rightarrow\mathbf{k}$. Recall that $\underline{\sigma}:\underline{X'}\rightarrow\underline{X}$ is an isomorphism and $\sM_{X'}\simeq\sM_X\oplus_{\sK_k}\sM_k$. One can construct the pushout $X'^{(1)}\amalg_{X'}X$ of the diagram
$$\xymatrix{
X' \ar[r]^{\sigma} \ar[d]^i & X\\
X'^{(1)} &
}$$
as follows:
\begin{itemize}
  \item The underlying scheme $\underline{X'^{(1)}\amalg_{X'}X}$ is defined to be $\underline{X'^{(1)}}$,
  \item the log structure of $X'^{(1)}\amalg_X'X$ is defined to be $\sM_{X'^{(1)}}\times_{\sM_{X'}}\sM_X$.
\end{itemize}
With the obvious assignments on morphisms, the pushout process along $\sigma:X'\rightarrow X$ gives a functor
$$
B: \textrm{Lift}(X')\rightarrow\textrm{Lift}(X).
$$
It is straightforward to check the following proposition.
\begin{prop}
The functor $A$ gives an  equivalence of groupoids, and the functor $B$ is its quasi-inverse.
\end{prop}
\end{rmk}

\section{Examples}
In this section, $k$ is an algebraically closed field of characteristic $p>0$. We proceed to construct examples of semi-stable reductions over $W$ whose special fibers do not admit log deformation to ${\bf W}_2$, which negate Problem \ref{prob_Illusie} because of Theorem \ref{thm_decom_lifting}.
\subsection{More preparations}
The first lemma is another characterization of semi-stable reductions over $W=W(k)$.
\begin{lem}\label{dejong lemma}
Let $K_0$ be the fractional field of $W$. Then an $W$-scheme $X$ is a semi-stable reduction over $W$ if and only if the following two properties hold:
\begin{enumerate}
	\item the generic fiber $X_{K_0}=X\times_W{K_0}$ is smooth over $K_0$,
	\item the special fiber $X_{k}=X\times_Wk$ is a normal crossing variety over $k$.
\end{enumerate}
\end{lem}
\begin{proof}
See \cite{deJong1996}, 2.16.
\end{proof}
The second lemma is rather standard.
\begin{lem}\label{lem_w2lifting}
Let $X/{\bf k}$ be a log variety of semi-stable type. Assume the irreducible components $\{X_i, i\in I\}$ of the underlining variety $\underline{X}$ to be smooth. Let $\sX$ be a smooth deformation $X$ over ${\bf W}_2$. Then the underlying scheme of $\sX$ is written into the schematic union of closed subschemes $\underline{\sX}=\bigcup_{i\in I} \sX_i$s with the property that, for each nonempty $J\subseteq I$, the schematic intersection $\bigcap_{j\in J}\sX_j$ is a $W_2$-lifting of $\bigcap_{j\in J}X_j$.
\end{lem}
\begin{proof}
Set $$\sI_i=I_i+pI_i,$$ where $I_i$ is the ideal sheaf of $X_i$ in $X$. Then, $\sI_i$ is an ideal sheaf of $\sO_{\underline{\sX}}$. We claim that the closed subschemes $\sX_i$s defined by $\sI_i$s have the property in the lemma. To show this it suffices to prove the following properties:
\begin{enumerate}
  \item $\sO_{\underline{\sX}}/\sI_i$ is flat over $W_2$,
  \item $\bigcap \sI_i=0$, and
  \item for each nonempty $J\subseteq I$, $\sO_{\underline{\sX}}/\cup_{j\in J}\sI_j$ is flat over $W_2$.
\end{enumerate}
Since $\widehat{\sO_{\underline{\sX},x}}$ is faithfully flat over $\sO_{\underline{\sX},x}$ for each point $x\in \underline{\sX}$, it suffices to verify the above claim after tensoring with $\widehat{\sO_{\underline{\sX},x}}$ for every $x\in\underline{\sX}$. By (\cite{KKato1988} Theorem 3.5, Proposition 3.14), there is an \'etale morphism $U\rightarrow \underline{\sX}$ such that we have
$$\xymatrix{
U\ar[r]^-f \ar[dr]_{\pi'|_{U}} &\textrm{Spec}(W_2[x_1,\cdots,x_n]/(x_1\cdots x_r)) \ar[d]\\
 & \spec(W_2)
},$$
where $f$ is an \'etale morphism. As a consequence, there is an isomorphism
$$\alpha:\widehat{\sO_{\underline{\sX},x}}\cong W_2[[x_1,\cdots,x_n]]/(x_1\cdots x_r)$$ such that each $\sI_i\widehat{\sO_{\underline{\sX},x}}$ (whenever it is nonempty) is generated by $\alpha^{-1}(\Pi_{j\in J_i}x_j)$ for some nonempty set $J_i\subseteq\{1,\cdots,r\}$. Moreover, $\{1,\cdots,r\}$ is the disjoint union of $J_i$s. Then the claim follows from direct calculations.
\end{proof}
By the above two lemmas, we can conclude the following
\begin{prop}\label{simple lemma}
Let $Z$ be a smooth scheme over $W$. Let $Y_0$ be a smooth closed subvariety of $Z_0=Z\times_Wk$. Set $X=Bl_{Y_0}Z$, the blowup of $Z$ along the closed subscheme $Y_0$. Then $X$ is a semi-stable reduction over $W$, whose special fiber $X_0$ is a simple normal crossing divisor consisting of two smooth components $Bl_{Y_0}Z_0$ and $\mathbb{P}(N_{Y_0/Z})$ (the projective normal bundle of $Y_0$ in $Z$) which intersect transversally along $\mathbb{P}(N_{Y_0/Z_0})$ (the projective normal bundle of $Y_0$ in $Z_0$). Furthermore, if the normal crossing variety $X_0$ over ${\bf k}$ admits a smooth deformation over ${\bf W}_2$, then both pairs $(Bl_{Y_0}Z_0,\mathbb{P}(N_{Y_0/Z_0}))$ and $(\mathbb{P}(N_{Y_0/Z}),\mathbb{P}(N_{Y_0/Z_0}))$ are $W_2(k)$-liftable.
\end{prop}
\begin{proof}
The first statement follows from Lemma \ref{dejong lemma} (the remaining fact is fairly standard and therefore omitted, see \cite{Fulton1998}). The second statement follows from Lemma \ref{lem_w2lifting}.
\end{proof}
\begin{prop}[Cynk-van Straten, \cite{Cynk2009} Theorem 3.1]\label{cynklemma}
	Let $\pi:Y\rightarrow X$ be a morphism of schemes over $k$ and let $S=\spec(A)$, where $A$ is artinian with residue field $k$. Assume that $\sO_X =\pi_\ast\sO_Y$ and $R^1\pi_\ast(\sO_Y)=0$. Then for every lifting $\sY\rightarrow S$ of $Y$ there exists a preferred lifting $\sX\rightarrow S$ making a commutative diagram
	$$\xymatrix{
		Y \ar@{^{(}->}[r]\ar[d] & \sY \ar[d] \\
		X \ar@{^{(}->}[r] & \sX
	}$$
\end{prop}

\begin{cor}\label{prop_conterexample2}
Notation as in Proposition \ref{simple lemma}. If $Y_0$ is not $W_2(k)$-liftable, then the special fiber $X_0$ of $X$ (regarded as a log variety over ${\bf k}$) does not admit any smooth deformation over ${\bf W}_2$.
\end{cor}
\begin{proof}
	Use Propositions \ref{simple lemma} and \ref{cynklemma}  which assert that the $W_2$-liftability of $\mathbb{P}(N_{Y_0/Z})$ implies that of $Y_0$.
\end{proof}

\subsection{Example 1}
Corollary \ref{prop_conterexample2} provides direct examples: take a smooth projective variety $Y_0$ over $k$ which is non $W_2$-liftable, and take a closed embedding $Y_0\hookrightarrow Z_0$ over $k$ into a smooth projective variety such that the codimension $\textrm{Cod}_{Z_0}Y_0\geq 2$ and $Z_0$ admits a smooth lifting $Z$ over $W$ (for example take $Z_0$ to be a projective space of high dimension). Set $X=\textrm{Bl}_{Y_0}Z$, the blowup of $Z$ along the closed subscheme $Y_0$.  Then $X$ is a semi-stable reduction over $W$ whose special fiber $X_0/{\bf k}$ does not admit ${\bf W}_2$-deformation.

\subsection{Example 2}
Notice that Mukai \cite{Mukai2013} has obtained a nice generalization to higher dimension of Raynaud's classical example \cite{Raynaud1978} of non $W_2$-liftable smooth projective surface over $k$. His construction, together with an idea of Liedtke-Satriano (Theorem 1.1 (a) \cite{LM2014}), allows us to make concrete examples of all relative dimensions $\geq 2$. \\

Let us recall first the following
\begin{defn}[\cite{Mukai2013}]\label{defn_Tango}
A smooth curve $C$ over $k$ of genus $\geq2$ is called a Tango-Raynaud curve if there exists a rational function $f$ on $C$ such that $df\neq0$ and that $(df)=pD$ for some ample divisor $D$.
\end{defn}
A typical example of Tango-Raynaud curve is the plane curve defined by the affine polynomial
$$G(x^p)-x=y^{pe-1},$$
where $G$ is a polynomial of degree $e\geq 1$ in the variable $x$. The following lemma is well known.
\begin{lem}[\cite{Mukai2013}]\label{lem_tango_curve}
Let $C$ be a Tango-Raynaud curve, then there exists a rank two vector bundle $E$ on $C$ together with a smooth curve $D$ in the projectification $\mathbb{P}_C(E)$ of $E/C$, such that the composite $D\rightarrow \mathbb{P}_C(E)\rightarrow C$ is the relative Frobenius $F_0: D\rightarrow D^{(p)}=C$.
\end{lem}
\begin{prop}\label{examples of all dimensions}
Notation as in Lemma \ref{lem_tango_curve}. Let $\sC$ be a $W$-lifting of $C$ and $\sE$ a lifting of $E$ over $\sC$. For $d\geq 2$, set $Z_d=\mathbb{P}_{\sC}(\sE\oplus \sO_{\sC}^{d-2})$ and $X_d=Bl_{D}Z_d$. Then $X_d$ is a semi-stable reduction over $W$ of relative dimension $d$, whose special fiber, regarded as a log variety over $\bf k$, is non ${\bf W}_2$-liftable and therefore DR-indecomposable.
\end{prop}
\begin{proof}
We prove the statement for $d=2$ only (the proof for $d\geq 3$ is the same). Denote
	$$
	C_0=C, \quad Y_0=D, \quad Z_0=\mathbb{P}_C(E), \quad Z=Z_2.
	$$
Assume the contrary that the special fiber $X_0$ of $Bl_{Y_0}Z$, regarded as a log variety over ${\bf k}$, admit a smooth deformation over ${\bf W}_2$. It follows from Proposition \ref{lem_w2lifting} that the pair $(Z_0, Y_0)$ consisting of the component $Z_0=Bl_{Y_0}Z_0$ of $X_0$ together with the divisor $Y_0=\mathbb{P}(N_{Y_0/Z})\cap Z_0\subset X_0$ lift to a pair $(Z_1,Y_1)$ over $W_2$ (The scheme $Z_1$ is not necessarily the mod $p^2$-reduction of $Z$). On the other hand, Proposition \ref{cynklemma} implies that the projection $Z_0\to C_0$ is the reduction of a certain $W_2$-morphism $Z_1\to C_1$. Therefore, the composite $F_0: Y_0\hookrightarrow Z_0\to C_0$ lifts to the composite $F_1: Y_1\hookrightarrow Z_1\to C_1$ over $W_2$. But this leads to a contradiction: the nonzero morphism $dF_1: F_1^*\Omega_{C_1}\to \Omega_{Y_1}$ is divisible by $p$ and it induces a nonzero morphism over $k$
	$$
	\frac{dF_1}{p}:F_0^*\Omega_{C_0}\to \Omega_{Y_0},
	$$
	which is impossible because of the degree. Therefore, $X_0/{\bf k}$ is indeed non ${\bf W}_2$-liftable as claimed.
\end{proof}

\section{An $E_1$-degeneration result}
This section is devoted to prove the following
\begin{thm}\label{thm_degeneration_E1}
Let $k$ be an algebraically closed field and $R$ a DVR with the residue field $k$. Let $Z/R$ be a smooth proper $R$-scheme and $X/R$ be a blow-up of $X$ along a closed regular center $Y_0$ supported in $Z_0=Z\times_Rk$. If the Hodge to de Rham spectral sequence
$$E_1^{pq}=H^q(Z_0,\Omega^p_{Z_0})\Rightarrow H^{p+q}(\Omega_{Z_0}^{\bullet})$$
degenerates at $E_1$ (e.g. when $\textrm{char}(k)=0$ or $\dim Z_0\leq \textrm{char}(k)$ and $R$ is of mixed characteristic), then the Hodge to log de Rham spectral sequence
$$E_1^{pq}=H^q(X_0,\wedge^p\Omega_{X_0}^{\log})\Rightarrow H^{p+q}(\Omega_{X_0}^{\log\bullet})$$
degenerates at $E_1$.
\end{thm}
Recall from Proposition \ref{simple lemma} that $X_0$ is a simple normal crossing divisor consisting of two smooth components $X_1=Bl_{Y_0}Z_0$ and $X_2=\mathbb{P}(N_{Y_0/Z})$ which intersect transversally along $D=\mathbb{P}(N_{Y_0/Z_0})$. The blowdown morphism of the log pairs $(Z,Z_0)\rightarrow(X,X_0)$ restricts on the special fiber to a log morphism $\pi:X_0\rightarrow (Z_0,1\mapsto 0)$ between log varieties over $(\spec(k),1\mapsto0)$. This induces a canonical morphism
$$\pi^{\ast i}:\Omega^i_{Z_0}\rightarrow R\pi_\ast\bigwedge^i\Omega_{X_0}^{\textrm{log}}.$$
Our main technical step in proving Theorem \ref{thm_degeneration_E1} is the following
\begin{prop}\label{prop_KEY}
Let $Z/R$ be a smooth proper $R$-scheme and $X/R$ be a blow-up of $X$ along a closed regular center $Y_0$ supported in $Z_0$. Denote by $\pi:X_0\rightarrow Z_0$ the restriction morphism. Then for each $i$ the canonical morphism (defined in the proof)
$$\Omega_{Z_0}^i \rightarrow R\pi_\ast\wedge^i\Omega_{X_0}^{\textrm{log}}$$
is an isomorphism in $D^b(Z_0)$.
\end{prop}
From Proposition \ref{prop_KEY}, we may derive the main result of the section.
\begin{proof}[Proof of Theorem \ref{thm_degeneration_E1}]
We actually prove that the two spectral sequences
\begin{align}\label{align_derham_1}
E_1^{pq}=H^q(Z_0,\Omega^p_{Z_0})\Rightarrow H^{p+q}(\Omega_{Z_0}^{\bullet})
\end{align}
and
\begin{align}\label{align_derham_2}
E_1^{pq}=H^q(X_0,\wedge^p\Omega_{X_0}^{\log})\Rightarrow H^{p+q}(\Omega_{X_0}^{\log\bullet})
\end{align}
are isomorphic. First recall that (\ref{align_derham_1}) is induced by the hypercohomology of the complex
$\Omega^{\bullet}_{Z_0}$
with respect to the truncated filtration
$$F^i=\tau^{\textrm{st}}_{\geq i}\Omega^{\bullet}_{Z_0},$$
where $\tau^{\textrm{st}}$ is the stupid truncation.
(\ref{align_derham_2}) is induced by the hypercohomology of the complex
$\Omega_{X_0}^{\textrm{log}\bullet}$
with respect to the truncated filtration
$$F^i=\tau^{\textrm{st}}_{\geq i}\Omega_{X_0}^{\textrm{log}\bullet}.$$
By Proposition \ref{prop_KEY}, there are natural quasi-isomorphisms
$$R\pi_{\ast}\Omega_{X_0}^{\log\bullet}\simeq\pi_{\ast}\Omega_{X_0}^{\log\bullet}\simeq\Omega^{\bullet}_{Z_0},$$
and the isomorphisms respect the filtration
$$F^i=R\pi_{\ast}\tau^{\textrm{st}}_{\geq i}\Omega_{X_0}^{\textrm{log}\bullet}\simeq\pi_{\ast}\tau^{\textrm{st}}_{\geq i}\Omega_{X_0}^{\textrm{log}\bullet}$$
in the left, middle and
$$F^i=\tau^{\textrm{st}}_{\geq i}\Omega^{\bullet}_{Z_0}$$
in the right. As a consequence, the two spectral sequences (\ref{align_derham_1}) and (\ref{align_derham_2}) are naturally isomorphic.
\end{proof}
To prove Proposition \ref{prop_KEY}, we make some preparations. Let $X_0=X_1\cup_DX_2$ be a variety consisting of two smooth projective components $X_1$ and $X_2$ such that they intersect transversely along a smooth divisor $D$. 
Assume that $X_0$ has a log structure of semi-stable type (Definition \ref{defn_semistable_type}). Then the normalization $X_1\cup X_2\rightarrow \underline{X_0}$ and the diagonal immersion $D\rightarrow X_1\cup X_2$ lift to log morphisms
$$(X_1\cup X_2, D_1\cup D_2\oplus(1\mapsto 0))\rightarrow X_0$$
and
$$(D,(1\mapsto 0)^{\oplus 2})\rightarrow (X_1\cup X_2, D_1\cup D_2)$$
over the base $(\spec(k),1\mapsto 0)$. These log morphisms induce morphisms of sheaves on $X_0$
\begin{align}\label{align_log_diff1}
\Omega^i_{X_0^{\textrm{log}}}\rightarrow \Omega_{X_1}^i(\textrm{log}D)\oplus \Omega_{X_2}^i(\textrm{log}D)
\end{align}
and
\begin{align}\label{align_log_diff2}
\Omega_{X_1}^i(\textrm{log}D)\oplus \Omega_{X_2}^k(\textrm{log}D)\rightarrow\Omega^k_{(D,(1\mapsto 0)^{\oplus 2})/(\spec(k),1\mapsto 0)}.
\end{align}
for each $i$. By the definition of log cotangent sheaf,
$$\Omega_{(D,(1\mapsto 0)^{\oplus 2})/(\spec(k),1\mapsto 0)}\simeq\Omega_D\oplus (\mathbb{Z}^{\oplus 2}/\mathbb{Z}\otimes_{\mathbb{Z}}\sO_{D})/\alpha(m)\otimes m-d\alpha(m)\otimes 1.$$
Thanks to the log structure of $(D,(1\mapsto 0)^{\oplus 2})$, $\alpha(m)\otimes m-d\alpha(m)\otimes 1$ are null relations. Therefore
$$\Omega_{(D,(1\mapsto 0)^{\oplus 2})/(\spec(k),1\mapsto 0)}\simeq\Omega_D\oplus \sO_D.$$
This isomorphism induces the forgetful morphism
$$\Omega_{(D,(1\mapsto 0)^{\oplus 2})/(\spec(k),1\mapsto 0)}\rightarrow \Omega_D$$
and the log residue morphism
$$\Omega_{(D,(1\mapsto 0)^{\oplus 2})/(\spec(k),1\mapsto 0)}\rightarrow \sO_D.$$
Therefore
$$\Omega^k_{(D,(1\mapsto 0)^{\oplus 2})/(\spec(k),1\mapsto 0)}\cong \bigwedge^k(\Omega_D\oplus \sO_D)\simeq \Omega_D^k\oplus \Omega_D^{k-1}$$
and by local calculation the restriction morphism
$$\Omega_{X_1}^k(\textrm{log}D)\rightarrow\Omega^k_{(D,(1\mapsto 0)^{\oplus 2})/(\spec(k),1\mapsto 0)}$$
is equivalent to
$$(\iota,\textrm{res}_D):\Omega_{X_1}^k(\textrm{log}D)\rightarrow \Omega_D^k\oplus \Omega_D^{k-1}$$
$$\beta+\gamma \frac{dz}{z}\mapsto (\beta,\gamma).$$
Here we use a local chart of $X_1$ where $D=\{z=0\}$ and $\beta$, $\gamma$ does not contain $dz$.
This phenomenon is interesting in itself. The residue map $\textrm{res}_D$ is a part of the restriction map of log cotangent sheaves. It makes log geometry a convenient and natural language in such a situation.\\

Assume locally $X_0$ is embedded into the affine space with a system of local coordinates $(z_1,z_2,\cdots,z_n)$ such that $X_1=\{z_1=0\}$, $X_2=\{z_2=0\}$. Since $$\frac{dz_1}{z_1}+\frac{dz_2}{z_2}=0$$ on $X_0$, a log form on $X_0$ is of the form
$$\beta+\gamma_1\frac{dz_1}{z_1}=\beta-\gamma_1\frac{dz_2}{z_2}.$$
Therefore two $k$-forms $\beta_1+\gamma_1\frac{dz_1}{z_1}$ on $X_2$ and $\beta_2+\gamma_2\frac{dz_2}{z_2}$ on $X_2$ glue to a log $k$-form on $X_0$ if and only if
$$\beta_1|_D=\beta_2|_D$$
and
$$\gamma_1|_D+\gamma_2|_D=0.$$
This proves
\begin{lem}\label{lem_log_cotangent_normalization}
For each $k\geq 0$ there is a short exact sequence of sheaves
$$0\rightarrow\Omega_{X^{\textrm{log}}}^k\rightarrow \Omega_{X_1}^k(\textrm{log}D)\oplus \Omega_{X_2}^k(\textrm{log}D)\stackrel{\varphi}{\rightarrow} \Omega_D^{i}\oplus \Omega_D^{i-1}\rightarrow 0$$
where $\varphi$ is defined by
$$\begin{pmatrix}
\iota & \textrm{res}_D \\
-\iota & \textrm{res}_D
\end{pmatrix}.$$
Here $\Omega_D^{-1}$ is defined to be $0$.
\end{lem}
The following well-known lemma will be used several times in the sequel.
\begin{lem}\label{lem_Bott}
Let $\mathbb{P}^n$ be the projective space over $k$. The following vanishing results hold:
\begin{enumerate}
  \item
  \begin{align*}
  H^q(\mathbb{P}^n,\Omega_{\mathbb{P}^n}^p)=0,  p\neq q
  \end{align*}
  \item If $i\neq0$, then
  \begin{align*}
  H^q(\mathbb{P}^n,\Omega_{\mathbb{P}^n}^p(i))=0,
  \end{align*}
  for $q=0$, $i\leq p$ or $q=n$, $i\geq p-n$ or $q\neq0,n$.
  \item Let $H$ be a hyperplane in $\mathbb{P}^n$, then
  \begin{align*}
  H^q(\mathbb{P}^n,\Omega_{\mathbb{P}^n}^p(\textrm{log}H))=
  \begin{cases}
  k, & p=q=0 \\
  0, & \textrm{otherwise}.
  \end{cases}
  \end{align*}
\end{enumerate}
\end{lem}
\begin{lem}\label{lem_proj_bundle}
Let $Z$ be a smooth variety and $\pi:P\rightarrow Z$ be a projective bundle of relative dimension $r$. Let $D\subset P$ be a relative hyperplane. Then for each $i\geq 0$ there is a canonical isomorphism
$$\Omega^i_Z\simeq R\pi_\ast\Omega^i_P(\textrm{log}D)$$
in $D(Z)$.
\end{lem}
\begin{proof}
The exact sequence
$$0\to \pi^\ast\Omega_Z\to \Omega_P(\textrm{log}D)\to \Omega_{P/Z}(\textrm{log}D)\to0$$
induces a decreasing filtration
$$F^p=\pi^{\ast}\Omega^p_{Z}\wedge\Omega_P^{i-p}(\textrm{log}D)\subset\Omega_P^{i}(\textrm{log}D)$$
such that $$F^p/F^{p+1}\simeq\pi^{\ast}\Omega^p_{Z}\otimes\Omega_{P/Z}^{i-p}(\textrm{log}D).$$
Therefore we have a spectral sequence
$$E_1^{pq}=R^q\pi_{\ast}(\pi^{\ast}\Omega^p_{Z}\otimes\Omega_{P/Z}^{i-p}(\textrm{log}D))\Rightarrow R^{p+q}\pi_\ast(\Omega^{i}_P(\textrm{log}D)).$$
By Lemma \ref{lem_Bott}, we see that
  \begin{align*}
  E_1^{pq}\simeq \Omega^p_{Z}\otimes R^q\pi_{\ast}(\Omega_{P/Z}^{i-p}(\textrm{log}D))=
  \begin{cases}
  \Omega^i_{Z}, & p=i, q=0 \\
  0, & \textrm{otherwise}.
  \end{cases}
  \end{align*}
  This proves the lemma.
\end{proof}

\begin{lem}\label{lem_blowup}
Let $Z_0$ be a smooth projective variety and $Y_0$ be a smooth closed subvariety of $Z_0$. Denote $\pi:X_1\rightarrow Z_0$ be the blowup along $Y_0$ with exceptional divisor $D$. Then for each $k\geq0$, there is a distinguished triangle in $D^b(Z_0)$ induced by natural morphisms:
$$\Omega^k_{Z_0}\stackrel{u}{\rightarrow} R\pi_\ast\Omega_{X_1}^k\oplus\Omega_{Y_0}^k\stackrel{v}{\rightarrow} R\pi_\ast\Omega_D^k\rightarrow\Omega^k_{Z_0}[1].$$
In other words, we have the short exact sequence
\begin{align}\label{align_isomorphism_1}
0\rightarrow\Omega^k_{Z_0}\rightarrow \pi_\ast\Omega_{X_1}^k\oplus\Omega_{Y_0}^k\rightarrow \pi_\ast\Omega_D^k\rightarrow 0
\end{align}
and the isomorphism
\begin{align}\label{align_isomorphism_2}
R^i\pi_\ast\Omega_{X_1}^k\rightarrow R^i\pi_\ast\Omega_D^k
\end{align}
for each $i>0$.
\end{lem}
\begin{proof}
Denote the following automorphism of $\pi_\ast\Omega_{X_1}^k\oplus\Omega_{Y_0}^k$ by $\phi$:
$$
(a,b)\mapsto (a,a-b),
$$
By composing with $\phi$, the exactness of the sequence (\ref{align_isomorphism_1}) is reduced to the following isomorphisms
$$
\Omega^k_{Z_0}\cong \pi_\ast\Omega_{X_1}^k;\quad \Omega_{Y_0}^k\cong \pi_\ast\Omega_D^k.
$$
For $k=0$, these are obvious. For $k\geq 1$, their truth can be easily seen by considering the local model of a blow-up along a smooth center: we assume that $X_1$ is the blow up of $Z_0=\A^n$ along $Y_0=\A^r$ defined by the intersection of some coordinate hyperplanes. Then the map
$$
\pi: D\to Y_0
$$
is the projection
$$
\A^r\times \P^{s}\to \A^r.
$$
Thus, it is trivial to get $\pi_{*}\Omega^k_{D}=\Omega^k_{Y_0}, k\geq 0$ by this description. For the first isomorphism, we use the following estimation:
$$
\pi^*\Omega^k_{Z_0}\subset \Omega^k_{X_1}\subset \pi^*\Omega^k_{Z_0}(kD).
$$
From this, it follows that
$$
\Omega^k_{Z_0}\subset \pi_* \Omega^k_{X_1}\subset \Omega^k_{Z_0}\otimes \pi_*\sO_{X_1}(kD)= \Omega^k_{Z_0},
$$
and hence $\pi_* \Omega^k_{X_1}=\Omega^k_{Z_0}$. \\

The proof of (\ref{align_isomorphism_2}) is divided into two parts. First we show that the natural map
$$R^i\pi_\ast\Omega_{X_1}^k|_{D}\rightarrow R^i\pi_\ast\Omega_D^k$$
is an isomorphism for each $i>0$. Considering the long exact sequence associated to
$$0\rightarrow\sO_D(1)\otimes\Omega^{k-1}_D\rightarrow\Omega^k_{X_1}|_D\rightarrow\Omega^k_D\rightarrow0$$
where $\sO_D(1)$ is the tautological bundle of the projective bundle $D\rightarrow Y_0$, we see that it sufficient to prove that
\begin{align}\label{align_vanish_1}
R^i\pi_\ast(\sO_D(1)\otimes\Omega^{k}_D)=0,\quad i>0.
\end{align}
Notice that the short exact sequence
$$0\rightarrow\pi^{\ast}\Omega_{Y_0}\rightarrow\Omega_D\rightarrow\Omega_{D/Y_0}\rightarrow 0$$
induces a decreasing filtration
$$F^p=\pi^{\ast}\Omega^p_{Y_0}\wedge\Omega_D^{k-p}\subset\Omega_D^{k}$$
such that $$F^p/F^{p+1}\simeq\pi^{\ast}\Omega^p_{Y_0}\otimes\Omega_{D/Y_0}^{k-p}.$$
Therefore we have a spectral sequence
$$E_1^{pq}=R^q\pi_{\ast}(\pi^{\ast}\Omega^p_{Y_0}\otimes\Omega_{D/Y_0}^{k-p}\otimes\sO_D(1))\Rightarrow R^{p+q}\pi_\ast(\sO_D(1)\otimes\Omega^{k}_D).$$
Since $D\rightarrow Y_0$ is a projective bundle, we obtain that
$$E_1^{pq}=\Omega^p_{Y_0}\otimes R^q\pi_{\ast}(\Omega_{D/Y_0}^{k-p}\otimes\sO_D(1))$$
for $p+q\geq 1$ and $p,q\geq 0$, thanks to the Lemma \ref{lem_Bott}. This proves (\ref{align_vanish_1}) and thus
$$R^i\pi_\ast\Omega_{X_1}^k|_{D}\rightarrow R^i\pi_\ast\Omega_D^k$$
is an isomorphism for each $i>0$.\\

Next we show that the canonical morphism
$$R^i\pi_\ast\Omega_{X_1}^k\rightarrow R^i\pi_\ast(\Omega_{X_1}^k|_{D})$$
is an isomorphism for each $i>0$. By the long exact sequence associated to
$$0\rightarrow\Omega_{X_1}^k\otimes\sO_{X_1}(-D)\rightarrow\Omega_{X_1}^k\rightarrow\Omega_{X_1}^k|_{D}\rightarrow 0,$$
we see that it sufficient to show the vanishing
\begin{align}\label{align_vanish2}
R^i\pi_{\ast}(\Omega_{X_1}^k\otimes\sO_{X_1}(-D))=0
\end{align}
for each $i>0$.\\

Notice that the short exact sequence
$$0\rightarrow\pi^{\ast}\Omega_{Z_0}\rightarrow\Omega_{X_1}\rightarrow\Omega_{X_1/Z_0}\simeq\Omega_{D/Y_0}\rightarrow 0$$
induces a decreasing filtration
$$F^p=\pi^{\ast}\Omega^p_{Z_0}\wedge\Omega_{X_1}^{k-p}\subset\Omega_{X_1}^{k}$$
such that $$F^p/F^{p+1}\simeq\pi^{\ast}\Omega^p_{Z_0}\otimes\Omega_{D/Y_0}^{k-p}.$$
Therefore we have a spectral sequence
$$E_1^{pq}=R^q\pi_{\ast}(\pi^{\ast}\Omega^p_{Z_0}\otimes\Omega_{D/Y_0}^{k-p}\otimes\sO_{X_1}(-D))\Rightarrow R^{p+q}\pi_\ast(\Omega^{k}_{X_1}\otimes\sO_{X_1}(-D)).$$
Since $D\rightarrow Y_0$ is a projective bundle, we obtain that
$$E_1^{pq}=\Omega^p_{Z_0}\otimes R^q\pi_{\ast}(\Omega_{D/Y_0}^{k-p}\otimes\sO_D(1))$$
for $p+q\geq 1$ and $p,q\geq 0$, thanks again to the Lemma \ref{lem_Bott}. This proves (\ref{align_vanish2}) and thus
$$R^i\pi_\ast\Omega_{X_1}^k\rightarrow R^i\pi_\ast(\Omega_{X_1}^k|_{D})$$
is an isomorphism for each $i>0$. So we finish the proof of (\ref{align_isomorphism_2}).
\end{proof}
Now we are ready to prove Proposition \ref{prop_KEY}.
\begin{proof}

By Lemma \ref{lem_log_cotangent_normalization} and \ref{lem_proj_bundle}, we have a distinguished triangle
$$R\pi_\ast\bigwedge^i\Omega^{\textrm{log}}_{X}\rightarrow R\pi_\ast\Omega_{X_1}^i(\textrm{log}D)\oplus\Omega_{Y_0}^i\stackrel{R\pi_\ast\varphi}{\rightarrow} R\pi_\ast\Omega_D^{i}\oplus R\pi_\ast\Omega_D^{i-1}\rightarrow R\pi_\ast\Omega_{X^{\textrm{log}}}^i[1]$$
in $D^b(Z_0)$. This triangle fills in the following diagram in $D^b(Z_0)$
\begin{align}\label{align_33}
\xymatrix{
R\pi_\ast\Omega^i_{D} \ar[r] & 0 \ar[r] & R\pi_\ast\Omega_D^i[1] \ar[r]^{\textrm{Id}} & R\pi_\ast\Omega_D^i[1]\\
R\pi_\ast\Omega^{i}_{X_1}\oplus\Omega_{Y_0}^i \ar[u] \ar[r] & R\pi_\ast\Omega_{X_1}^i(\textrm{log}D)\oplus\Omega_{Y_0}^i \ar[u] \ar[r] & R\pi_\ast\Omega_D^{i-1} \ar[u] \ar[r] & R\pi_\ast\Omega^{i}_{X_1}[1]\oplus\Omega_{Y_0}^i[1] \ar[u]\\
R\pi_\ast\wedge^i\Omega^{\textrm{log}}_{X_0} \ar[u]^{p} \ar[r] & R\pi_\ast\Omega_{X_1}^i(\textrm{log}D)\oplus\Omega_{Y_0}^i \ar[u]^{\textrm{Id}} \ar[r]^-{R\pi_\ast\varphi} & R\pi_\ast\Omega_D^{i}\oplus R\pi_\ast\Omega_D^{i-1} \ar[u]^{\textrm{pr}} \ar[r] & R\pi_\ast\wedge^i\Omega^{\textrm{log}}_{X_0}[1] \ar[u]\\
R\pi_\ast\Omega^i_{D}[-1] \ar[r] \ar[u] & 0 \ar[r] \ar[u] & R\pi_\ast\Omega^i_{D} \ar[u] \ar[r]^{\textrm{Id}} & R\pi_\ast\Omega^i_{D}, \ar[u]
}
\end{align}
which is generated from the centered commutative square
$$\xymatrix{
R\pi_\ast\Omega_{X_1}^i(\textrm{log}D)\oplus\Omega_{Y_0}^i \ar[r] & R\pi_\ast\Omega_D^{i-1} \\
R\pi_\ast\Omega_{X_1}^i(\textrm{log}D)\oplus\Omega_{Y_0}^i \ar[u]^{\textrm{Id}} \ar[r]^-{R\pi_\ast\varphi} & R\pi_\ast\Omega_D^{i}\oplus R\pi_\ast\Omega_D^{i-1}. \ar[u]\\
}$$

In the diagram (\ref{align_33}), $p$ is induced (non-canonically) by the above commutative square. The second horizontal line is the direct sum of the distinguished triangles
$$R\pi_\ast\Omega^{i}_{X_1}\rightarrow R\pi_\ast\Omega_{X_1}^i(\textrm{log}D)\rightarrow  R\pi_\ast\Omega_D^{i-1}\rightarrow R\pi_\ast\Omega^{i}_{X_1}[1]$$
and
$$\Omega_{Y_0}^i\stackrel{\textrm{Id}}{\rightarrow} \Omega_{Y_0}^i\rightarrow 0\rightarrow\Omega_{Y_0}^i[1].$$
The horizontal lines of (\ref{align_33}) are distinguished triangles. The second and third vertical lines are also distinguished. By the $3\times3$ lemma of triangulated categories, the first vertical line induces a distinguished triangle
$$R\pi_\ast\wedge^i\Omega^{\textrm{log}}_{X_0}\rightarrow R\pi_\ast\Omega_{X_1}^k\oplus\Omega_{Y_0}^i\rightarrow R\pi_\ast\Omega_D^k\rightarrow R\pi_\ast\wedge^i\Omega^{\textrm{log}}_{X_0}[1].$$
Comparing with Lemma \ref{lem_blowup}, we see that there is a quasi-isomorphim
$$R\pi_\ast\wedge^i\Omega_{X_0}^{\textrm{log}}\simeq\Omega_{Z_0}^i.$$
Note that this isomorphism may not be the natural one induced by the morphism $\pi$. However, we obtain as a consequence of the abstract quasi-isomorhism that
$$R^k\pi_\ast\wedge^i\Omega_{X_0}^{\textrm{log}}\simeq 0,\quad k>0.$$
It remains to show that the natural morphism of sheaves
\begin{align}\label{align_KEY0}
\Omega_{Z_0}^i \rightarrow \pi_\ast\wedge^i\Omega_{X_0}^{\textrm{log}}
\end{align}
is an isomorphism.\\

Let us consider the cohomologies at place 0 of the diagram (\ref{align_33}),
\begin{align}\label{align_330}
\xymatrix{
\pi_\ast\Omega^{i}_{X_1}\oplus\Omega_{Y_0}^i \ar[r] & \pi_\ast\Omega_{X_1}^i(\textrm{log}D)\oplus\Omega_{Y_0}^i \ar[r] & \pi_\ast\Omega_D^{i-1}& \\
\pi_\ast\wedge^i\Omega^{\textrm{log}}_{X_0} \ar[u]^{p^0} \ar[r] & \pi_\ast\Omega_{X_1}^i(\textrm{log}D)\oplus\Omega_{Y_0}^i \ar[u]^{\textrm{Id}} \ar[r]^-{\pi_\ast\varphi} & \pi_\ast\Omega_D^{i}\oplus R\pi_\ast\Omega_D^{i-1} \ar[u]^{\textrm{pr}} \ar[r] & 0\\
0 \ar[u]\ar[r]  & 0 \ar[r] \ar[u] & \pi_\ast\Omega^i_{D} \ar[u] \ar[r]^{\textrm{Id}} & \pi_\ast\Omega^i_{D} \ar[u]
}.
\end{align}
The two vertical sequences in the middle are short exact sequences. Therefore, by the snake lemma, there is an exact sequence
$$0\to \pi_\ast\wedge^i\Omega^{\textrm{log}}_{X_0}\stackrel{p^0}{\to}\pi_\ast\Omega^{i}_{X_1}\oplus\Omega_{Y_0}^i\stackrel{\delta}{\to}\pi_\ast\Omega^i_{D}$$
where $\delta$ is the boundary map which is identical to the one in (\ref{align_isomorphism_1}). Hence by (\ref{align_isomorphism_1}) we see that the natural map (\ref{align_KEY0}) is an isomorphism.
\end{proof}

\textbf{Acknowledgment:} We would like to thank Luc Illusie for several valuable e-mail communications, and Weizhe Zheng for another argument to Theorem \ref{thm_decom_lifting} which is also included in the note. Warm thanks go to Christian Liedtke for his interest and comments. The first named author would like to thank Kang Zuo for his interest and constant support. The second named author would like to express his deep gratitude to Xiaotao Sun and Jun Li for their constant encouragement throughout the work.


\begin{bibdiv}
\begin{biblist}
\bib{Cynk2009}{article}{
   author={Cynk, S{\l}awomir},
   author={van Straten, Duco},
   title={Small resolutions and non-liftable Calabi-Yau threefolds},
   journal={Manuscripta Math.},
   volume={130},
   date={2009},
   number={2},
   pages={233--249},
   issn={0025-2611},
   review={\MR{2545516}},
   doi={10.1007/s00229-009-0293-0},
}
\bib{FKato1996}{article}{
   author={Kato, Fumiharu},
   title={Log smooth deformation theory},
   journal={Tohoku Math. J. (2)},
   volume={48},
   date={1996},
   number={3},
   pages={317--354},
   issn={0040-8735},
   review={\MR{1404507}},
   doi={10.2748/tmj/1178225336},
}
\bib{KKato1988}{article}{
   author={Kato, Kazuya},
   title={Logarithmic structures of Fontaine-Illusie},
   conference={
      title={Algebraic analysis, geometry, and number theory (Baltimore, MD,
      1988)},
   },
   book={
      publisher={Johns Hopkins Univ. Press, Baltimore, MD},
   },
   date={1989},
   pages={191--224},
   review={\MR{1463703 (99b:14020)}},
}
\bib{deJong1996}{article}{
   author={de Jong, A. J.},
   title={Smoothness, semi-stability and alterations},
   journal={Inst. Hautes \'Etudes Sci. Publ. Math.},
   number={83},
   date={1996},
   pages={51--93},
   issn={0073-8301},
   review={\MR{1423020}},
}
\bib{Del_Ill1987}{article}{
   author={Deligne, Pierre},
   author={Illusie, Luc},
   title={Rel\`evements modulo $p^2$ et d\'ecomposition du complexe de de
   Rham},
   language={French},
   journal={Invent. Math.},
   volume={89},
   date={1987},
   number={2},
   pages={247--270},
   issn={0020-9910},
   review={\MR{894379}},
   doi={10.1007/BF01389078},
}
\bib{Illusie2002}{book}{
   author={Bertin, Jos{\'e}},
   author={Demailly, Jean-Pierre},
   author={Illusie, Luc},
   author={Peters, Chris},
   title={Introduction to Hodge theory},
   series={SMF/AMS Texts and Monographs},
   volume={8},
   note={Translated from the 1996 French original by James Lewis and
   Peters},
   publisher={American Mathematical Society, Providence, RI; Soci\'et\'e
   Math\'ematique de France, Paris},
   date={2002},
   pages={x+232},
   isbn={0-8218-2040-0},
   review={\MR{1924513 (2003g:14009)}},
}

\bib{Friedman1983}{article}{
   author={Friedman, Robert},
   title={Global smoothings of varieties with normal crossings},
   journal={Ann. of Math. (2)},
   volume={118},
   date={1983},
   number={1},
   pages={75--114},
   issn={0003-486X},
   review={\MR{707162 (85g:32029)}},
   doi={10.2307/2006955},
}

\bib{Fulton1998}{book}{
   author={Fulton, William},
   title={Intersection theory},
   series={Ergebnisse der Mathematik und ihrer Grenzgebiete. 3. Folge. A
   Series of Modern Surveys in Mathematics [Results in Mathematics and
   Related Areas. 3rd Series. A Series of Modern Surveys in Mathematics]},
   volume={2},
   edition={2},
   publisher={Springer-Verlag, Berlin},
   date={1998},
   pages={xiv+470},
   isbn={3-540-62046-X},
   isbn={0-387-98549-2},
   review={\MR{1644323}},
   doi={10.1007/978-1-4612-1700-8},
}
\bib{Illusie1990}{article}{
   author={Illusie, Luc},
   title={R\'eduction semi-stable et d\'ecomposition de complexes de de Rham
   \`a\ coefficients},
   language={French},
   journal={Duke Math. J.},
   volume={60},
   date={1990},
   number={1},
   pages={139--185},
   issn={0012-7094},
   review={\MR{1047120}},
   doi={10.1215/S0012-7094-90-06005-3},
}

\bib{LM2014}{article}{
	author={Liedtke, Christian},
	author={Satriano, Matthew},
	title={On the birational nature of lifting},
	journal={Adv. Math.},
	volume={254},
	date={2014},
	pages={118--137},
	issn={0001-8708},
	review={\MR{3161094}},
	doi={10.1016/j.aim.2013.10.030},
}

\bib{Lang1995}{article}{
	author={Lang, William E.},
	title={Examples of liftings of surfaces and a problem in de Rham cohomology},
	journal={Compositio Math.},
	volume={97},
	number={1-2}
	date={1995},
	pages={157--160},}
	
\bib{Mukai2013}{article}{
   author={Mukai, Shigeru},
   title={Counterexamples to Kodaira's vanishing and Yau's inequality in
   positive characteristics},
   journal={Kyoto J. Math.},
   volume={53},
   date={2013},
   number={2},
   pages={515--532},
   issn={2156-2261},
   review={\MR{3079312}},
   doi={10.1215/21562261-2081279},
}

\bib{Nakkajima2000}{article}{
   author={Nakkajima, Yukiyoshi},
   title={Liftings of simple normal crossing log $K3$ and log Enriques
   surfaces in mixed characteristics},
   journal={J. Algebraic Geom.},
   volume={9},
   date={2000},
   number={2},
   pages={355--393},
   issn={1056-3911},
   review={\MR{1735777}},
}

\bib{Ogus2014}{book}{
   author={Ogus, Arthur},
   title={Lectures on Logarithmic Algebraic Geometry},
   note={Texed notes},
   date={2014},
}

\bib{Ollson2003}{article}{
	author={Olsson, Martin C.},
	title={Universal log structures on semi-stable varieties},
	journal={Tohoku Math. J. (2)},
	volume={55},
	date={2003},
	number={3},
	pages={397--438},
	issn={0040-8735},
	review={\MR{1993863}},
}

\bib{Raynaud1978}{article}{
   author={Raynaud, M.},
   title={Contre-exemple au ``vanishing theorem'' en caract\'eristique
   $p>0$},
   language={French},
   conference={
      title={C. P. Ramanujam---a tribute},
   },
   book={
      series={Tata Inst. Fund. Res. Studies in Math.},
      volume={8},
      publisher={Springer, Berlin-New York},
   },
   date={1978},
   pages={273--278},
   review={\MR{541027}},
}

\end{biblist}
\end{bibdiv}
\end{document}